\documentclass{proc-l}

\usepackage{amssymb}

\usepackage{physics} %for a lot
\usepackage{bbm} %mathbbm{1}
\usepackage{mathtools} %coloneqq
\usepackage{mleftright} %for \mleft \right bracket spacing (use \mleft and \mright)
\usepackage[dvipsnames]{xcolor} %for citation edits; remove once done
\usepackage{cleveref} %Cref

\usepackage{enumitem} %for (i) in enumerate

\copyrightinfo{}{}

\newtheorem{theorem}{Theorem}
\newtheorem{lemma}[theorem]{Lemma}
\newtheorem{proposition}[theorem]{Proposition}

\theoremstyle{definition}
\newtheorem*{definition}{Definition}

\theoremstyle{remark}
\newtheorem*{remark}{Remark}

\numberwithin{equation}{section}

\newcommand*\de{\mathrm{d}}
\DeclareMathOperator{\R}{\mathbb{R}}

\DeclareMathOperator{\Z}{\mathbb{Z}}
\DeclareMathOperator{\T}{\mathbb{T}}

\DeclareMathOperator{\Rea}{Re}
\DeclareMathOperator{\Ima}{Im}

\newcommand{\innerprod}[2]{\langle #1, #2 \rangle}
\begin{document}

\title[Almost sure growth bounds for NLS and \lowercase{g}K\lowercase{d}V]{On growth of Sobolev norms for periodic non-\\linear Schrödinger and generalised Korteweg-de Vries equations under critical Gibbs dynamics}


%    author one information
\author{Fabian Höfer}
\address{Mathematics Münster, University of Münster, Orléans-Ring 10, 48149 Münster, Germany}
%\curraddr{}
\email{fabian.hoefer@uni-muenster.de}
\thanks{F.H. was funded by the Deutsche Forschungsgemeinschaft (DFG, German Research Foundation) under Germany's Excellence Strategy EXC 2044 –390685587, Mathematics Münster: Dynamics–Geometry–Structure.}

%    author two information
\author{Niko A.\ Nikov}
\address{School of Mathematics, The University of Edinburgh, James Clerk Maxwell Building, Peter Guthrie Tait Road, Edinburgh, EH9 3FD, United Kingdom}
%\curraddr{}
\email{n.a.nikov@sms.ed.ac.uk}
%\thanks{}

%    \subjclass is required.

\date{}

\dedicatory{}

%    Abstract is required.
\begin{abstract}
	We prove logarithmic growth bounds on Sobolev norms of the focusing mass-critical NLS and gKdV equations on the torus, which hold almost surely under the focusing Gibbs measure with optimal mass threshold constructed by Oh, Sosoe, and Tolomeo in [Invent.\ Math.\ 227 (2022), no.\ 3, 1323--1429]. More precisely, we will establish almost sure growth bounds for solutions $u(t)$ of the equations of the form 
        \[
            \sup_{t \in [-T,T]} \norm{u(t)}_{H^s(\T)} \lesssim_{s, u_0} \log(2+T)
        \]
        with initial data $u_0 \in H^s(\T)$ for $s< \frac{1}{2}$. The proof uses a generalisation of Bourgain's invariant measure argument for measures in a suitable Orlicz space.
\end{abstract}

\maketitle
\thispagestyle{plain}

\section{Introduction} \label{Introduction}
	We consider Cauchy problems for the focusing nonlinear Schrödinger equation
	\begin{equation}\tag{NLS}\label{NLS}
		  \begin{cases} 
                &i \partial_t u + \partial_x^2 u + \abs{u}^{p-2} u = 0 \\
                &u(0)=u_0,
            \end{cases}
	\end{equation}
    and the focusing generalised Korteweg-de Vries equation
	\begin{equation}\tag{gKdV}\label{gKdV}
		  \begin{cases}
		      &\partial_t u + \partial_x^3 u + \partial_x (u^{p-1})= 0 \\
                &u(0)=u_0
		  \end{cases}
	\end{equation}
    for $p > 2$ posed on the one-dimensional torus $\T = \R / \Z$. Our focus will be on mass-critical equations corresponding to $p=6$. We will adopt the point of view of statistical mechanics, dating back to the seminal paper of Lebowitz, Rose, and Speer \cite{LRS}, who initiated the study of focusing Gibbs measures for \eqref{NLS}. Later, Bourgain, in \cite{Bourgain94}, studied dynamics under the Gibbs measure; there is a rich literature in this direction: see e.g.\ \cite{BL23,Bri23,BDNY24,CLO24, DNY22, DNY24, DR23,FT22,LW24,Rob24,ST20}. Using ideas from \cite{Bourgain94}, we obtain almost sure growth bounds of $H^s(\T)$ norms, $s < \frac12$, for \eqref{NLS} and \eqref{gKdV} under their respective Gibbs measure at the optimal mass threshold, as built by Oh, Sosoe, and Tolomeo in \cite{OST}. See \Cref{a.s. bounds NLS}. \medskip
    
    We start by reformulating \eqref{NLS} and \eqref{gKdV} as Hamiltonian systems:
	\[
		\partial_t u = - i \frac{\delta H_\mathrm{NLS}}{\delta \bar{u}}, \qquad H_\mathrm{NLS}(u) = \frac{1}{2} \int_{\T} \abs{\partial_x u}^2 \, \de x - \frac{1}{p} \int_{\T} \abs{u}^p \, \de x
	\]
	for \eqref{NLS}, and
	\[
		\partial_t u = \partial_x \frac{\delta H_\mathrm{gKdV}}{\delta u}, \qquad H_\mathrm{gKdV}(u) = \frac{1}{2} \int_{\T} (\partial_x u)^2 \, \de x - \frac{1}{p} \int_{\T} u^p \, \de x
	\]
	for \eqref{gKdV}.
    Here, $\frac{\delta}{\delta u}$ and $\frac{\delta}{\delta \bar{u}}$ denote Fréchet derivatives. In view of the Hamiltonian structures, we expect Gibbs measures $\rho$, which are probability measures formally given by
	\[
		\de \rho (u) = \frac{1}{Z} e^{-H(u)}\, \de u,
	\]
	to be invariant under the dynamics of the equations. Here, $H$ is a Hamiltonian functional and $Z$ denotes a normalisation constant, called the partition function. By ``invariant'' we mean that $\rho\circ\Phi(-t)=\rho$ for all $t \in \R$, where $\Phi(t)$ is the flow of each equation. The defocusing case corresponds to a negative (resp.\ positive) sign in front of the nonlinearity in the equation (resp.\ in the Hamiltonian), and the defocusing Gibbs measure corresponds to the well-studied $\Phi^p$-measure from quantum field theory. We note that $H_{\mathrm{NLS}}$ is essentially the complex version of $H_{\mathrm{gKdV}}$, which is reflected in the Gibbs measures. Inserting $H_\mathrm{NLS}$ for the Hamiltonian, we arrive at the following formal density for the focusing Gibbs measure for \eqref{NLS}:
	\begin{equation}\label{formal Gibbs measure}
		\de \rho_p (u)  = \frac{1}{Z_p} e^{\frac{1}{p} \int_{\T} \abs{u}^p \, \de x - \frac{1}{2} \int_{\T} \abs{\partial_x u}^2 \, \de x} \, \de u.
	\end{equation}
    Since there is no infinite-dimensional analogue $\de u$ of the Lebesgue measure, as a first step to rigorously defining the Gibbs measure, we write it as a density with respect to a Gaussian:
    \begin{equation}\label{untruncated Gibbs measure}
        \de\rho_p (u)=\frac{1}{Z_p} e^{\frac{1}{p} \int_{\T} \abs{u}^p \, \de x} \,\de\mu_0(u).
    \end{equation}
    Here, $\mu_0$ is the Gaussian measure with inverse covariance $- \Delta$, formally given by
    \begin{equation}\label{Massless Gaussian}
        \de \mu_0(u) = \frac{1}{Z} e^{- \frac{1}{2} \int_{\T} \abs{\partial_x u}^2\, \de x}\, \de u.       
    \end{equation}
    The measure can be seen as the law of the random Fourier series
	\[
	   \sum_{n \in \Z\setminus\{0\}} \frac{g_n}{2 \pi n} e^{2\pi i n x}, 
	\]
    taking values in the space of periodic distributions $\mathcal{D}'(\T)$, where $\{g_n\}_{n \in \Z \setminus \{0\}}$ denotes a sequence of independent standard complex-valued Gaussian random variables.\footnote{I.e.\ $\Rea g_n$ and $\Ima g_n$ are independent real-valued mean-zero Gaussian random variables with variance $\frac{1}{2}$. In the real-valued setting of \eqref{gKdV}, we need to impose $g_{-n} = \overline{g_n}$ for all $n \in \Z \setminus \{0\}$.}
    An easy calculation shows that the measure $\mu_0$ is supported on $\bigcap_{s<\frac{1}{2}}H^s(\T)$. In the focusing case, $H$ is not bounded from below, and one can show that the density $e^{\frac{1}{p} \int_{\T} \abs{u}^p \, \de x}$ in \eqref{untruncated Gibbs measure} is not integrable with respect to $\mu_0$. For this reason, \cite{LRS} considered the focusing Gibbs measure with a mass cutoff,\footnote{In the defocusing case,  $H$ is bounded from below and thus no mass cutoff is needed.} i.e.
	\begin{equation}\label{truncated Gibbs measure mean-zero}
		\de \rho_{p,K} (u)  = \frac{1}{Z_{p,K}} e^{\frac{1}{p} \int_{\T} \abs{u}^p \, \de x} \mathbbm{1}_{\{M(u) \leq K\}} \, \de \mu_0(u),
	\end{equation}
	where $M(u) = \int_{\T} \abs{u}^2 \, \de x$ denotes the mass and $K>0$. Recalling that the Hamiltonian $H$ and the mass $M$ are conserved under the NLS and gKdV dynamics, we expect \eqref{truncated Gibbs measure mean-zero} to be invariant measures whenever they exist. \medskip
    
    We recall the normalisability of $\rho_{p,K}$; see \cite{OST} for details. Let $Q$ be the unique even and positive optimiser for the Gagliardo-Nirenberg-Sobolev inequality on $\R$ (see Appendix B in \cite{Tao} for details), namely
	\begin{equation}\label{GNS} \tag{GNS}
		\norm{u}_{L^p(\R)}^p \le C_{\rm GNS}(p) \norm{u}_{L^2(\R)}^{2 + \frac{p-2}{2}}\norm{\partial_x u}_{L^2(\R)}^{\frac{p-2}{2}}
	\end{equation}
	at $p=6$, where $\norm{Q}_{L^6(\R)}^6 = 3 \norm{\partial_x Q}_{L^2(\R)}^{2}$. Then $\rho_{p,K}$ is normalisable (i.e.\ $Z_{p,K}<\infty$) exactly for $2<p<6$ and $K>0$, or $p=6$ and $K\le \norm{Q}_{L^2(\R)}^2$. This was proven for the regimes $2<p<6$, $K>0$, and $p=6$, $K<\norm{Q}_{L^2(\R)}^2$, as well as $p=6$, $K>\norm{Q}_{L^2(\R)}^2$ in \cite{LRS}, although the proof contained a gap, and was corrected later in \cite{Bourgain94} for $p=6$ and $K>0$ sufficiently small. The proof of the normalisability and invariance at $p=6$ and $K< \norm{Q}_{L^2(\R)}^2$ and even at the optimal threshold when $K= \norm{Q}_{L^2(\R)}^2$ was completed in \cite{OST}.\footnote{This result was somewhat unexpected due to the existence of minimal mass blow-up solutions $u_*$ to \eqref{NLS} with $\norm{u_*}_{L^2(\T)} = \norm{Q}_{L^2(\R)}$, constructed by Ogawa and Tsutsumi in \cite{OgawaTsutsumi}.} There are also several works that study Gibbs measures for \eqref{NLS} in the infinite volume limit; see e.g.\ \cite{Rid02,SS24,TolomeoWeber}. \medskip

    In the following, we will focus on the case $p = 6$ and use the notation $\rho = \rho_{6,K_0}$ with $K_0=\norm{Q}_{L^2(\R)}^2$. Using Bourgain's invariant measure argument introduced in \cite{Bourgain94} (see \Cref{Bourgain's strategy}), one can show that, for $K < \norm{Q}_{L^2(\R)}^2$ and $s < \frac12$, almost surely with respect to $\rho_{6,K}$, there exists a constant $C(s,u_0)>0$ such that 
    \begin{equation} \label{NLS GWP}
        \sup_{t \in [-T,T]} \norm{u(t)}_{H^s(\T)} \leq C(s,u_0) \sqrt{\log(2+T)}
    \end{equation}
    and one has $\rho_{6,K}$-a.s.\ global well-posedness of \eqref{NLS} in $\bigcap_{s < \frac{1}{2}} H^{s}(\T)$. However, the constant $C(s,u_0)$ obtained from Bourgain's strategy blows up as $M(u_0)$ increases to $\norm{Q}_{L^2(\R)}^2$ as a consequence of the fact that the density $\frac{\de\rho}{\de\mu}$ is not in $L^p(\mu)$ for any $p > 1$ (see \Cref{density is not in Lp}). This does not affect the $\rho$-a.s.\ global well-posedness, as the shell $\{ M(u_0)=\norm{Q}_{L^2(\R)}^2 \}$ has zero $\rho$-measure; see Corollary 1.6 in \cite{OST} and Appendix C in \cite{ChapoutoKishimoto} for \eqref{gKdV}. Our main result improves this global well-posedness in the form of a growth bound as in \eqref{NLS GWP}, where we trade the divergent constant for a power of the logarithm. The similarity of $H_\mathrm{NLS}$ and $H_\mathrm{gKdV}$ allows us to prove a counterpart in the \eqref{gKdV} case.
    \begin{theorem}\label{a.s. bounds NLS}
		Let $p = 6$ and $K = \norm{Q}_{L^2(\R)}^2$. 
        \begin{enumerate}[label = \normalfont{(\roman*)}] \label{}
            \item \textnormal{(A.s.\ bounds for \eqref{NLS}).} Let $H = H_{\mathrm{NLS}}$. Then, for $\rho$-almost every $u_0\in\bigcap_{s<\frac12}H^s(\T)$, there exists a global-in-time solution $u \in \bigcap_{s < \frac{1}{2}} C(\R ; H^{s}(\T))$ to the equation \eqref{NLS} with initial data $u_0$ which satisfies
		  \[
		  	\sup_{t \in [-T,T]} \norm{u(t)}_{H^s(\T)} \leq C(s, u_0) \log(2+T)
		  \]
		  for some constant $C(s,u_0) >0$ and all $s < \frac{1}{2}$ and $T>0$.
          \item \textnormal{(A.s.\ bounds for \eqref{gKdV}).} Let $H = H_{\mathrm{gKdV}}$. Then, for $\rho$-almost every $u_0\in\bigcap_{s<\frac12}H^s(\T)$, there exists a global-in-time solution $u \in \bigcap_{s < \frac{1}{2}} C(\R ; H^{s}(\T))$ to the equation \eqref{gKdV} with initial data $u_0$ which satisfies
		\[
			\sup_{t \in [-T,T]} \norm{u(t)}_{H^s(\T)} \leq C(s, u_0) \log(2+T)
		\]
		for some $C(s,u_0)>0$ and all $s < \frac{1}{2}$ and $T>0$.
        \end{enumerate}
	\end{theorem}
    \begin{remark}
        In \cite{ChapoutoKishimoto}, Chapouto and Kishimoto used Bourgain's argument and the results from \cite{OST} to prove $\rho$-a.s.\ global well-posedness for \eqref{gKdV} in Fourier-Lebesgue spaces. As a byproduct, we will improve this result by showing $\rho$-a.s.\ growth bounds 
        \[
            \sup_{t \in [-T,T]} \norm{u(t)}_{\mathcal{F}L^{s,p}(\T)} \leq C(s,p,u_0) \log(2+T),
        \]
        where $2 < p < \infty$, $s < 1 - \frac1p$ and $T>0$. See \Cref{gKdVcase} for details.
    \end{remark}
	\begin{remark}
		In the proofs we choose to replace $\mu_0$ by the Gaussian measure $\mu$ with inverse covariance $1 - \Delta$, formally given by 
        \begin{equation}\label{Massive Gaussian}
			\de \mu(u) = \frac{1}{Z} e^{-\frac{1}{2} \int_{\T} \abs{u}^2\, \de x - \frac{1}{2} \int_{\T} \abs{\partial_x u}^2 \, \de x} \, \de u,
		\end{equation}
        which can be seen as the law of the random Fourier series\footnote{Here, $\langle n \rangle \coloneqq (1 + |n|^2)^{\frac12}$.}
        \[
            \sum_{n \in \Z} \frac{g_n}{\langle n \rangle}e^{2\pi i n x}.
        \]
		The focusing Gibbs measure with optimal mass cutoff then reads
		\begin{equation}\label{truncated Gibbs measure}
			\de \rho (u)  = \frac{1}{Z_K} e^{\frac{1}{6} \int_{\T} \abs{u}^6 \, \de x} \mathbbm{1}_{\{M(u) \leq \norm{Q}_{L^2(\R)}^2\}} \, \de \mu(u).
		\end{equation}
		As noted in \cite{OST} and seen in \cite{Bourgain94}, $\mu$ is a more natural base Gaussian measure, due to the lack of conservation of the spatial mean under the dynamics of \eqref{NLS}.
		In either case, all the results from \cite{OST} are valid in both the complex value setting and the real value setting (with restriction $g_{-n} = \overline{g_n}$) and with both the measures \eqref{Massless Gaussian} and \eqref{Massive Gaussian} as base Gaussian measures; see Remark 1.2 in \cite{OST}. Also, all cited results from \cite{ChapoutoKishimoto} work with both base Gaussians as well, see Remark 1.7 in \cite{ChapoutoKishimoto}. Hence, we will only work with the Gibbs measure \eqref{truncated Gibbs measure} in the complex-valued setting. 
	\end{remark}
    The remainder of the paper is organised as follows. In \Cref{Bourgain's strategy}, we demonstrate why Bourgain's strategy is not directly applicable. In \Cref{Orlicz}, we introduce Orlicz spaces and construct a Young function, which will be used in \Cref{Our strategy} to overcome this obstacle. The remainder of the proof, following closely the ideas of \cite{OST}, is carried out in \Cref{Int int}. In \Cref{gKdVcase}, we sketch the proof in the case of \eqref{gKdV}.
    
    \section{Bourgain's strategy and a roadblock} \label{Bourgain's strategy}
    In \cite{Bourgain94}, Bourgain outlined the following argument to replace the use of deterministic conservation laws with measure invariance to obtain logarithmic growth bounds. Using finite-dimensional approximations at which level one has global well-posedness, this yields almost-sure global well-posedness for the full equation. We have the following local well-posedness result for \eqref{NLS}.
	\begin{theorem}[\cite{Bourgain93}, Theorem 1]\label{LWP_NLS}
    	Let $s>0$. Then \eqref{NLS} is locally well-posed in $H^s(\T)$. More precisely, there exists $\beta >0$ such that for all $M \geq 1$ and all $u_0 \in H^s(\T)$ with  $\norm{u_0}_{H^s(\T)} \leq M$ there exists a time $\tau=\tau(M) \gtrsim M^{-\beta}$ and a solution\footnote{Uniqueness holds in a suitable $X^{s,b}$ space.} $u \in C([-\tau,\tau]; H^s(\T))$ to \eqref{NLS} with $u(0) = u_0$, and $\sup_{t \in [-\tau,\tau]} \norm{u(t)}_{H^s(\T)} \leq 2M$.
        Moreover, the solution $u$ depends continuously on the initial data $u_0$.
	\end{theorem}
    Assume that we are given an invariant measure $\nu$ for \eqref{NLS} and assume that $\nu$ is absolutely continuous with respect to $\mu$ with Radon-Nikodym derivative $\frac{\de\nu}{\de\mu}$ in $L^p(\mu)$ for some $p>1$. We can then run Bourgain's argument as follows: fixing $M\geq 1$ we have, by the local well-posedness \Cref{LWP_NLS}, the invariance of $\nu$ under the \eqref{NLS} flow, and Hölder's inequality,
    \begin{align*}
        \nu\mleft( \sup_{t \in [-T,T]} \norm{u(t)}_{H^s(\T)} \geq 2M \mright)
        &\leq
        \nu\mleft( \sup_{|j| \leq \mleft\lfloor \frac{T}{\tau(M)} \mright\rfloor} \norm{u(j \cdot \tau(M))}_{H^s(\T)} \geq M \mright)
        \\
        &\leq
        \sum_{|j| \leq \mleft\lfloor \frac{T}{\tau(M)} \mright\rfloor} \nu( \norm{u(j \cdot \tau(M))}_{H^s(\T)} \geq M)
        \\
        &=
        \left( 1 + 2\mleft\lfloor \frac{T}{\tau(M)} \mright\rfloor \right) \nu(\norm{u_0}_{H^s(\T)} \geq M)
        \\
        &\leq
        \left( 1 + 2\mleft\lfloor \frac{T}{\tau(M)} \mright\rfloor \right) \norm{\frac{\de\nu}{\de\mu}}_{L^p(\mu)} \mu(\norm{u_0}_{H^s(\T)} \geq M)^{\frac{1}{p'}}
    \end{align*}
    for any $s < \frac{1}{2}$. Now, by a Gaussian tail estimate and again the local well-posedness \Cref{LWP_NLS}, we obtain
    \[
        \nu\mleft( \sup_{t \in [-T,T]} \norm{u(t)}_{H^s(\T)} \geq 2M \mright) \lesssim (1 + 2T M^{\beta}) \norm{\frac{\de\nu}{\de\mu}}_{L^p(\mu)} e^{- \frac{M^2}{4p'}}.  
    \]
    Using the Borel-Cantelli lemma, one can now prove the following $\nu$-a.s.\ bound:
    \[
        \sup_{t \in [-T,T]} \norm{u(t)}_{H^s(\T)} \lesssim_{s, u_0} \sqrt{\log(2+T)}.
    \]
    We omit this argument and present an analogous one in our case of interest. \medskip

    Naturally, we would like to adapt Bourgain's approach to the setting $\nu = \rho$, with $\rho$ as in \Cref{Introduction} (with $p=6$ and $K = \norm{Q}_{L^2(\T)}^2$). However, at this mass threshold, we have $\frac{\de\rho}{\de\mu}\not\in L^p(\mu)$ for any $p>1$ (see \Cref{density is not in Lp}) and we cannot apply Hölder's inequality in the above argument. Our approach will be to show an intermediate integrability; see \eqref{goal1}.
    \begin{proposition} \label{density is not in Lp}
        Let $p>1$. Then
        \[
        \int_{\mathcal{D}'(\T)} e^{\frac{p}{6}\int_{\T} \abs{u}^6\,\de x}\mathbbm{1}_{\{M(u)\le\norm{Q}_{L^2(\R)}^2\}}\, \de\mu(u)=\infty.
        \]
    \end{proposition}
    \begin{proof} 
        Identifying $\T$ with $[-\frac12,\frac12)$, for some small $\varepsilon > 0$ and $\delta>0$ to be chosen later, take
        \[
            Q_\delta^\chi(x) \coloneqq \chi(x)\delta^{-\frac12}Q(\delta^{-1}x), \qquad \widetilde{Q}_\delta^\chi(x) \coloneqq (1-\varepsilon)Q_\delta^\chi(x),
        \]
        where $\chi \colon \R \to [0,1]$ is a radially-decreasing smooth cutoff function with support in $[-\frac14,\frac14]$ and equal to $1$ on $[-\frac18,\frac18]$.
        We have the following elementary estimates using smoothness and exponential decay of $Q$ and $\partial_x Q$ (see Proposition B.7 in \cite{Tao}).
        \begin{lemma} \label{QM estimates}
            Let $0 < \delta \ll 1$. Then, for any
            $r \geq 1$ and any $c>0$, we have, uniformly in $0 < \delta \ll 1$,
            \begin{align}
                &\big\lVert\widetilde{Q}_\delta^\chi\big\rVert_{L^r(\T)}^r =(1-\varepsilon)^r \norm{Q}_{L^r(\R)}^r\delta^{\frac{2-r}2} + \mathcal{O}_{c,\chi}(\delta^{c}); 
                \\
                &\big\lVert\partial_x \widetilde{Q}_\delta^\chi\big\rVert_{L^2(\T)}^2 \le (1-\varepsilon)^2(1+\varepsilon)\norm{\partial_x Q}_{L^2(\R)}^2\delta^{-2} + \mathcal{O}_{\chi,\varepsilon}(1). 
            \end{align}
        \end{lemma}
        \noindent Let $\tau_w \colon \mathcal{D}'(\T) \to \mathcal{D}'(\T)$ be the translation operator $u\mapsto u-w$. By the Cameron-Martin formula, we have the following\footnote{Here and throughout the paper we view $H^s(\T)$ as a real vector space with inner product given by $\innerprod{f}{g}_{H^s(\T)} = \Rea \int_{\T} (1 - \Delta)^s f(x) \overline{g(x)} \, \de x$.}
        \begin{align*}
            \int_{\mathcal{D}'(\T)} & e^{\frac p6\int_{\T} \abs{u}^6\, \de x}\mathbbm{1}_{\{M(u)\le\norm{Q}_{L^2(\R)}^2\}}\, \de\mu(u) 
            \\
            &= 
            \int_{\mathcal{D}'(\T)} e^{\frac p6\lVert v-\widetilde{Q}_\delta^\chi\rVert_{L^6(\T)}^6} \mathbbm{1}_{\{M(v-\widetilde{Q}_\delta^\chi)\le\norm{Q}_{L^2(\R)}^2\}}\, \de\mu\circ\tau_{-\widetilde{Q}_\delta^\chi}(v) 
            \\
            &= 
            \int_{\mathcal{D}'(\T)} e^{\frac p6\lVert v-\widetilde{Q}_\delta^\chi\rVert_{L^6(\T)}^6 - \innerprod{\widetilde{Q}_\delta^\chi}{v}_{H^1(\T)}-\frac12\lVert\widetilde{Q}_\delta^\chi\rVert_{H^1(\T)}^2} \mathbbm{1}_{\{M(v-\widetilde{Q}_\delta^\chi)\le\norm{Q}_{L^2(\R)}^2\}}\, \de\mu(v).
        \end{align*}
        Using \Cref{QM estimates}, if $\norm{v}_{L^2(\T)}\le\frac\varepsilon2\norm{Q}_{L^2(\R)}$ and $0 < \delta \ll 1$ is sufficiently small depending on $\varepsilon$, then $\big\lVert v-\widetilde{Q}_\delta^\chi\big\rVert_{L^2(\T)}\le\norm{Q}_{L^2(\R)}$.
        This implies
        \begin{align*}
            \int_{\mathcal{D}'(\T)} &e^{\frac p6\int_{\T}\abs{u}^6\, \de x}\mathbbm{1}_{\{M(u)\le\norm{Q}_{L^2(\R)}^2\}}\, \de\mu(u) 
            \\
            &\ge
            \int_{\mathcal{D}'(\T)} e^{\frac p6\lVert v-\widetilde{Q}_\delta^\chi\rVert_{L^6(\T)}^6-\innerprod{\widetilde{Q}_\delta^\chi}{v}_{H^1(\T)}-\frac12\lVert\widetilde{Q}_\delta^\chi\rVert_{H^1(\T)}^2} \mathbbm{1}_{\{\norm{v}_{L^2(\T)}\le\frac\varepsilon2\norm{Q}_{L^2(\R)}\}}\, \de\mu(v).
        \end{align*}
        We can compute
        \begin{align*}
            \int_{\mathbb{T}} |v-\widetilde{Q}_\delta^\chi|^6 \, \de x -\innerprod{\widetilde{Q}_\delta^\chi}{v}_{H^1(\T)} 
            &= \int_{\mathbb{T}} ( |v|^2 - 2 \Re(v) \widetilde{Q}_\delta^\chi + (\widetilde{Q}_\delta^\chi)^2)^3 \, \de x  -\innerprod{\widetilde{Q}_\delta^\chi}{v}_{H^1(\T)}
            \\
            &= \big \lVert\widetilde{Q}_\delta^\chi \big \rVert_{L^6(\T)}^6 + P(v,\widetilde{Q}_\delta^\chi) - R(v,\widetilde{Q}_\delta^\chi),
        \end{align*}
        where $P(v,\widetilde{Q}_\delta^\chi) \geq 0$ and $R$ is odd in $v$. Hence, moving the indicator function inside the integral and using Jensen's inequality, we have
        \begin{align*}
            \int_{\mathcal{D}'(\T)} &e^{\frac p6\int_{\T}\abs{u}^6\, \de x} \mathbbm{1}_{\{M(u)\le\norm{Q}_{L^2(\R)}^2\}}\, \de\mu(u) 
            \\
            &\ge 
            \int_{\mathcal{D}'(\T)} e^{\mleft(\frac p6\lVert\widetilde{Q}_\delta^\chi\rVert_{L^6(\T)}^6-\frac12\lVert\widetilde{Q}_\delta^\chi\rVert_{H^1(\T)}^2-R(v,\widetilde{Q}_\delta^\chi)\mright)\mathbbm{1}_{\{\norm{v}_{L^2(\T)}\le\frac\varepsilon2\norm{Q}_{L^2(\R)}\}}} \, \de\mu(v)-1 
            \\
            &\ge 
            e^{\int_{\mathcal{D}'(\T)} \mleft(\frac p6\lVert\widetilde{Q}_\delta^\chi\rVert_{L^6(\T)}^6-\frac12\lVert\widetilde{Q}_\delta^\chi\rVert_{H^1(\T)}^2-R(v,\widetilde{Q}_\delta^\chi)\mright)\mathbbm{1}_{\{\norm{v}_{L^2(\T)}\le\frac\varepsilon2\norm{Q}_{L^2(\R)}\}}\, \de\mu(v)}-1 
            \\
            &\ge
            e^{\varepsilon_1\mleft(\frac p6\lVert\widetilde{Q}_\delta^\chi\rVert_{L^6(\T)}^6-\frac12\lVert\widetilde{Q}_\delta^\chi\rVert_{H^1(\T)}^2\mright)}-1,
        \end{align*}
        where in the last line, we used the fact that $v\mapsto R(v,\widetilde{Q}_\delta^\chi)$ is odd and $\mu$ is symmetric about $0$, and that the event $\{ \norm{v}_{L^2(\T)} \le \frac\varepsilon2\norm{Q}_{L^2(\R)} \}$ has some positive probability $\varepsilon_1$. Appealing to \Cref{QM estimates}, we finally obtain
        \begin{align*}
            \int_{\mathcal{D}'(\T)}& e^{\frac p6\int_{\T}\abs{u}^6\, \de x} \mathbbm{1}_{\{M(u)\le\norm{Q}_{L^2(\R)}^2\}}\, \de\mu(u) 
            \\
            &\ge
            e^{\varepsilon_1\mleft(\frac p6(1-\varepsilon)^6\norm{Q}_{L^6(\R)}^6\delta^{-2}-\frac12 (1- \varepsilon)^2 (1+\varepsilon)\norm{\partial_x Q}_{L^2(\R)}^2\delta^{-2}-C_\varepsilon\mright)} 
            \\
            &\geq
            e^{\varepsilon_1\mleft(\frac{1}{2} \norm{\partial_x Q}_{L^2(\R)}^2 (p(1-\varepsilon)^6 - (1- \varepsilon)^2 (1+\varepsilon) )\delta^{-2}-C_\varepsilon\mright)},
        \end{align*}
        and, recalling $p>1$, we can pick $\varepsilon$ sufficiently small to complete the proof by taking $\delta\to 0$. 
    \end{proof}
    We will circumvent the lack of $L^p$-integrability of $\frac{\de\rho}{\de\mu}$ by proving that this density lies in some space of integrability stronger than $L^1$ but weaker than $L^p$, $p>1$, and then using a suitable substitute for Hölder's inequality. This gain in integrability will suffice to ensure bounds of the form
    \[
        \sup_{t \in [-T,T]} \norm{u(t)}_{H^s(\T)} \lesssim_{s, u_0} \log(2+T).
    \]
    The space we use will be an Orlicz space with a suitable Young function, introduced in the next section.
    \section{Orlicz spaces and the construction of a Young function} \label{Orlicz}
    This section partly follows \cite{GrafakosII}, and we omit elementary proofs. Let $(\Omega,\mathcal{F},\nu)$ be a $\sigma$-finite measure space.
    \begin{definition}
        We call a function $\Psi \colon \R \to [0,\infty)$ a \emph{Young function} if it is an even, continuous, increasing on $[0,\infty)$, convex function with $\Psi(0)=0$ and $\lim_{x\to\infty}\frac{\Psi(x)}{x}=\infty$.
    \end{definition}
    We recall that the Legendre transform of $\Psi$ at $t \in \R$ is defined by
    \[
        \Psi^*(t) \coloneqq \sup_{s \in \R} st - \Psi(s) \in [0,\infty),
    \]
    where the final condition in the above definition ensures the finiteness of $\Psi^*$. It is easy to see that $\Psi$ is a Young function if and only if $\Psi^*$ is a Young function. 
    
    Next, we define the so-called the \emph{Orlicz space}.
    \begin{definition}
        Let $\Psi$ be a Young function. For a measurable function $f \colon \Omega \to \R$ we define
        \[
            \norm{f}_{\Psi(L)} \coloneqq \inf \left\{ c>0 : \int_\Omega \Psi\mleft(\frac{f(\omega)}{c}\mright) \, \de \nu(\omega) \leq 1 \right\},
        \]
        with the convention that $\inf \emptyset \coloneqq \infty$. 
        Furthermore, we define the \emph{Orlicz space}
        \[
            \Psi(L) \coloneqq \Psi(L)(\Omega, \mathcal{F}, \nu) \coloneqq \left\{ f \text{ meas.} : \exists\, a > 0 \text{ s.t.} \int_\Omega \Psi(a f(\omega)) \, \de \nu(\omega) < \infty  \right\}.
        \]
    \end{definition}
    \noindent One can check that after identifying functions that are equal $\mu$-almost everywhere, $\norm{\cdot}_{\Psi(L)}$ defines a norm. Note that $\Psi \colon x\mapsto |x|^p$ returns the spaces $\Psi(L)=L^p$.
    Importantly, Orlicz spaces enjoy a Hölder-type inequality: for all $f \in \Psi(L)$ and all $g \in \Psi^*(L)$ we have
        \begin{equation}\label{Orlicz Holder}
            \int_\Omega \abs{fg} \, \de \nu \leq 2 \norm{f}_{\Psi(L)} \norm{g}_{\Psi^*(L)}.
        \end{equation}
    The ``intermediate integrability'' mentioned at the end of \Cref{Bourgain's strategy} will, in fact, follow from the construction of a suitable Young function. In particular, we set 
    \begin{equation} \label{F defn}
        F(x) = \widetilde{F}(|x|), \qquad \widetilde{F} = (G - G(e^{1-\alpha}))_+, \qquad G(x) = x e^{q|\log x|^\alpha}
    \end{equation}
    for $x>0$, where $q > 0$ and $0 < \alpha < 1$ are chosen later, and $F(0) \coloneqq 0$.
    \begin{lemma}\label{F is Young}
        The function $F \colon \R \to [0,\infty)$ defined above is a Young function.
    \end{lemma}
    \begin{proof}
        It is easy to see that $F(0) = 0$, $F$ is even, continuous, increasing on $[0,\infty)$ and $\lim_{x \to \infty} \frac{F(x)}{x} = \infty$. To check that $F$ is convex we compute for $x>1$
        \begin{equation*}
            G'(x)
            =
            e^{q (\log x)^\alpha}\left(1 + q \alpha (\log x)^{\alpha-1}\right)
        \end{equation*}
        and for $x \geq e^{1- \alpha}$
        \begin{align*}
            G''(x)
            &=
            e^{q (\log x)^\alpha} q \alpha  \frac{(\log x)^{\alpha-2}}{x} 
            \left(
            \log x \left(1 + q \alpha (\log x)^{\alpha-1}\right) + \alpha - 1 \right)
            \\
            &\geq
            e^{q (\log x)^\alpha } q \alpha  \frac{(\log x)^{\alpha-2}}{x}
            q \alpha (1- \alpha)^\alpha > 0.
        \end{align*}
        This shows that $G$ is convex on $(e^{1- \alpha}, \infty)$ and hence that $F$ is convex on $\R$.
    \end{proof}
	\section{Our strategy} \label{Our strategy}
	Following Bourgain, we have
	\[
	   \rho
       \mleft(
        \sup_{t\in[-T,T]}\norm{u(t)}_{H^s(\T)}\ge 2M
       \mright) 
       \leq
       (1 + 2T M^{\beta}) 
        \int_{\mathcal{D}'(\T)} \frac{\de\rho}{\de\mu}(u_0) \mathbbm{1}_{\{\norm{u_0}_{H^s(\T)}\ge M\}}\, \de\mu(u_0).
	\]
	Let $F$ be as in \eqref{F defn}. The following propositions are the crux of our argument.
    \begin{proposition} \label{better integrability}
    	Let $q>0$ and $\alpha\in(0,\frac{1}{2})$ or $\alpha=\frac{1}{2}$ and $0<q\ll1$ sufficiently small. Then, choosing $F$ as in \eqref{F defn}, we have
    	\[
    	   \norm{\frac{\de\rho}{\de\mu}}_{F(L)(\mu)}<\infty.
    	\]
	\end{proposition}
    \begin{proposition} \label{Orlicz tail bound}
        Let $q > 0 $, $\alpha \in (0,1)$ and $s < \frac{1}{2}$. Then, choosing $F$ as in \eqref{F defn}, there exists $c = c(\alpha, q)$ and $M_0 = M_0(\alpha,q,s)$ such that for $M\geq M_0$ 
        \[
            \big\lVert \mathbbm{1}_{\{\norm{u_0}_{H^s(\T)} \geq M\}} \big\rVert_{F^*(L)(\mu)} \lesssim_{\alpha,q,s} e^{-cM^{2 \alpha}}.
        \]
    \end{proposition}
	\Cref{better integrability} is proven in the following section and we delay the proof of \Cref{Orlicz tail bound} to the end of the present section. Using these propositions and \eqref{Orlicz Holder}, together with the local well-posedness \Cref{LWP_NLS} as in \Cref{Bourgain's strategy}, we have, for $M\ge M_0$
	\[
    	\rho\mleft(\sup_{t\in[-T,T]}\norm{u(t)}_{H^s(\T)}\ge 2M\mright) 
        \lesssim_{\alpha,q,s}
        (1 + 2T M^{\beta})e^{-cM^{2\alpha}}.
	\]
	Working at $\alpha=\frac{1}{2}$ and $0<q\ll1$, choosing $M = C\log(2+T)$ with $C>0$ chosen later, and restricting to large enough integers $T\ge T_0$, we see that
    \[
        \rho\mleft( \sup_{t \in [-T,T]} \norm{u(t)}_{H^s(\T)} \geq 2C\log(2+T) \mright) \lesssim \left(1 + 2C^\beta T(\log(2+T))^\beta\right) (2+T)^{- cC}. 
    \]
    Choosing $C$ such that $-cC<-2$, we have
    \[
        \sum_{T = T_0}^\infty \rho\mleft( \sup_{t \in [-T,T]} \norm{u(t)}_{H^s(\T)} \geq 2C\log(2+T) \mright)
        < \infty
    \]
    and so the Borel-Cantelli lemma implies that $\rho(A) = 0$ where $A$ is the event of all $u_0$ such that
    \[
        \sup_{t \in [-T,T]} \norm{u(t)}_{H^s(\T)} \geq 2C\log(2+T)
    \]
    for infinitely many integers $T \geq T_0$, where $u$ solves \eqref{NLS} with initial data $u_0$. This implies the $\rho$-a.s.\ bound for all $T>0$ claimed in \Cref{a.s. bounds NLS}-(i):
    \[
        \sup_{t \in [-T,T]} \norm{u(t)}_{H^s(\T)} \lesssim_{s, u_0} \log(2+T).
    \]
	As promised, we conclude this section with the proof of \Cref{Orlicz tail bound}.
    \begin{lemma}\label{boundF*}
        Let $q > 0 $ and $\alpha \in (0,1)$. Then there exist constants $c_{\alpha,q},y_{\alpha,q}\geq1$ depending on $\alpha$ and $q$ such that for $y \geq y_{\alpha,q}$
        \[
            F^* (y) \leq e^{c_{\alpha,q} (\log y)^\frac{1}{\alpha}}.
        \]
    \end{lemma}     
    \begin{proof}
        Since $F$ is even, we have for $y > 0$
        \begin{align*}
           F^*(y)
           &= \sup_{x \geq 0} xy - F(x) 
           \\
           &= \left(\sup_{x \in [0,e^{1-\alpha}]} xy - F(x) \right) \vee \left( \sup_{x \geq e^{1- \alpha}} xy - F(x) \right) 
           \\
           &\leq 
           y e^{1-\alpha} \vee \left( \sup_{x \geq e^{1- \alpha}} xy - x e^{q (\log x)^\alpha} + G\mleft(e^{1-\alpha}\mright) \right).
        \end{align*}
        Let $h_y(x) \coloneqq xy - x e^{q (\log x)^\alpha} + G\mleft(e^{1-\alpha}\mright)$.
        Then we have $h_y'(x) = 0$ for $x \geq e^{1- \alpha}$ if and only if $y = G'(x)$, where
        \begin{align*}
            G'(x)
            &=
            e^{q (\log x)^\alpha}\left(1 + q \alpha (\log x)^{\alpha-1}\right)
            \\
            &\geq 
            e^{q (1- \alpha)^\alpha}\left(1 + q \alpha (1 - \alpha)^{\alpha-1}\right)
            \eqqcolon 
            y_{\alpha,q} 
            \geq
            1
        \end{align*}
        using the monotonicity of $G'$ as seen in the proof of \Cref{F is Young}. Furthermore,
        \[
            h_y''(x) = -G''(x) < 0.
        \]
        Hence, every zero of $h_y'$ is a global maximum point of $h_y$. To estimate $F^*$, we note that $h_y'(x) = 0$ is equivalent to
        \[
            \log y = q (\log x)^\alpha + \log( 1 + q \alpha (\log x)^{\alpha-1}).
        \]
        Hence, for these $x \geq e^{1-\alpha}$ and $y > 0$
        \[
            q (\log x)^\alpha \leq \log y \leq c (\log x)^\alpha, 
        \]
        where $c = q + \log(1 + q \alpha (1- \alpha)^{\alpha - 1})\frac{1}{(1 - \alpha)^\alpha}$. Thus for $y \geq y_{\alpha,q}$, the supremum of $h_y(x)$ for $x \geq e^{1- \alpha}$ is attained at $x = x(y) \leq e^{(q^{-1} \log y )^\frac{1}{\alpha}}$ and consequently for $y \geq y_{\alpha,q}$
        \begin{align*}
            F^*(y)
            &\leq y e^{1-\alpha} \vee \left( \sup_{x \geq e^{1- \alpha}} xy - x e^{q (\log x)^\alpha} + G\mleft(e^{1-\alpha}\mright) \right)
            \\
            &\leq
            y e^{1-\alpha} \vee y e^{(q^{-1} \log y )^\frac{1}{\alpha}}
            \\
            &\leq
            y e^{ (q^{-1}\log y)^{\frac{1}{\alpha}}}
            \\
            &\leq
            e^{( 1 + \left(q^{-1}\right)^{\frac{1}{\alpha}}) \left( \log y \right)^{\frac{1}{\alpha}}}. \qedhere
        \end{align*}
    \end{proof}
    \begin{proof}[Proof of \Cref{Orlicz tail bound}.]
        We have, for sufficiently large $M>M_0$, where $M_0$ is determined below,
        \begin{align*}
            \big\lVert \mathbbm{1}_{\{\norm{u_0}_{H^s(\T)} \geq M\}} \big\rVert_{F^*(L)(\mu)}
            &=
            \inf \left\{ \lambda > 0 : \int F^*\mleft(\frac{\mathbbm{1}_{\{\norm{u_0}_{H^s(\T)} \geq M\}}}{\lambda}\mright) \, \de \mu(u_0) \leq 1 \right\}
            \\
            &=
            \inf \left\{ \lambda>0 : F^*\mleft(\lambda^{-1}\mright) \mu \mleft( \norm{u_0}_{H^s(\T)} \geq M \mright) \leq 1 \right\}
            \\
            &\leq
            \inf \left\{ \lambda>0 : F^*\mleft(\lambda^{-1}\mright)c_s e^{-\frac{M^2}{4}} \leq 1 \right\}
            \\
            &\leq 
            \inf \left\{ 0 < \lambda \leq y_{\alpha,q}^{-1} : c_s e^{c_{\alpha,q} \left(\log \lambda^{-1}\right)^\frac{1}{\alpha} - \frac{M^2}{4} } \leq 1 \right\},
        \end{align*}
        and where we used the fact that $F^*(0) = 0$, $F^*$ is increasing, the tail estimate (from which we have the constant $c_s$), and \Cref{boundF*}. Here $M_0$ is such that
        \[
            c_s e^{c_{\alpha,q} (\log y_{\alpha,q})^\frac{1}{\alpha}- \frac{M_0^2}{4} } = 1
        \]
        is satisfied. Next, we note that
        \[
             c_s e^{c_{\alpha,q} (\log \lambda^{-1})^\frac{1}{\alpha} - \frac{M^2}{4} } \leq 1
        \]
        is satisfied for 
        \[
            \lambda \geq e^{- \left( \frac{M^2}{4c_{\alpha, q}} - \frac{\log c_s}{c_{\alpha, q}} \right)^\alpha }
            .
        \]
        Since $0 < \alpha < 1$ there exists $c_1 = c_1(\alpha , q , s)$ such that
        \[
            e^{- ( \frac{M^2}{4c_{\alpha, q}} - \frac{\log c_s}{c_{\alpha, q}} )^\alpha } \leq c_1 e^{ - \left( 4 c_{\alpha,q}\right)^{-\alpha} M^{2\alpha}},
        \]
        and hence 
        \[
            \big\lVert \mathbbm{1}_{\{\norm{u_0}_{H^s(\T)} \geq M\}} \big\rVert_{F^*(L)(\mu)} \leq c_1 e^{-\left( 4 c_{\alpha,q}\right)^{-\alpha} M^{2 \alpha}}. \qedhere
        \]
    \end{proof}
	\section{Proving the intermediate integrability} \label{Int int}
	Let $q>0$ and $\alpha\in(0,\frac{1}{2})$ or $\alpha=\frac{1}{2}$ and $0<q\ll1$. In order to prove \Cref{better integrability}, it suffices to show the following integrability condition\footnote{It is easy to see that $\int_\Omega \Psi\mleft( f(\omega)\mright) \, \de \nu(\omega) < \infty$ for a measurable function $f$ and a Young function $\Psi$ implies $\norm{f}_{\Psi(L)} < \infty$.}
	\begin{equation}
	   \int_{\mathcal{D}'(\T)} e^{\frac16\norm{u}_{L^6(\T)}^6+q\norm{u}_{L^6(\T)}^{6\alpha}}\mathbbm{1}_{\{M(u)\le \norm{Q}_{L^2(\R)}^2\}}\, \de\mu(u)<\infty. 
       \label{goal1}
	\end{equation}
    Our approach is based entirely on that taken in \cite{OST}, where the authors prove \eqref{goal1} for $q=0$, decomposing their analysis near and far from an approximate soliton manifold of optimisers for the GNS inequality. It turns out that far from the manifold, data are not sharp for the GNS inequality in a quantitative way. 
	\subsection{The soliton manifold}
	The goal of this subsection is to build an approximate soliton manifold to which to compare $L^2(\T)$ functions. Given $\delta>0$ and $x_0\in\T$, define the $L^2$-invariant scaling and translation
	\begin{equation}
	   Q_{\delta,x_0}(x) = \delta^{-\frac12} Q(\delta^{-1}(x-x_0)),
	\end{equation}
	where $x-x_0$ is interpreted as an element of $\T$, where we use the identification $\T \cong [-\frac{1}{2},\frac{1}{2})$. In order to obtain a suitable periodisation of $Q$, we multiply by a smooth cutoff function: let $\chi$ be as in the proof of \Cref{density is not in Lp} and accordingly consider $Q^\chi_{\delta,x_0} \colon \T \to \R$ given by
	\begin{equation}
    	Q^\chi_{\delta,x_0}(x) = \chi(x-x_0) \delta^{-\frac12} Q(\delta^{-1}(x-x_0)).
	\end{equation}
	\begin{remark}
    	The set
    	\[
        	\mathcal{M}_{\delta^\ast}
            =
            \{
            e^{i\theta} Q^\chi_{\delta,x_0} : 0 < \delta < \delta^\ast , x_0 \in \T, \theta \in \R 
            \}
    	\]
    	can be viewed as a smooth manifold of dimension $3$ embedded in $H^1(\T)$. The cutoff considered is necessary for this to be true; see Remark 6.2 in \cite{OST}.
        \end{remark}
	\subsection{Stability of GNS optimisers}
	Let $\mathbf{P}_{\le k}$ (resp.\ $\mathbf{P}_0$, $\mathbf{P}_{\ne 0}$) denote the sharp Fourier projection onto frequencies $\{|n|\le 2^k\}$ (resp.\ $\{0\}$, $\{ n\ne 0 \}$) and $u_{\le k} = \mathbf{P}_{\le k}u$ (resp.\ $u_{=0}=\mathbf{P}_0u$, $u_{\ne 0}=\mathbf{P}_{\ne 0}u$). Now, we define for $\gamma>0$
	\begin{equation*} 
        \begin{split}
        	S_\gamma = \big\{ u \in L^2&(\T) : M(u) \le \norm{Q}_{L^2(\R)}^2,
            \\
        	&\norm{ \mathbf{P}_{\le k} u_{\ne 0}}_{L^6(\T)}^6 \le (C_{\rm GNS} - \gamma) \norm{\partial_x \mathbf{P}_{\le k}u}_{L^2(\T)}^2 \norm{Q}_{L^2(\R)}^4 \text{ for all } k\ge 1 \big\}.
	    \end{split}
    \end{equation*}
    As the elements in $S_\gamma$ are not sharp for the GNS inequality, we expect good integrability over this set. In fact, this is the case.
	\begin{lemma} \label{S_gamma integrability}
    	Let $\gamma>0$. Then there exists an exponent $1<r<\infty$, depending on $\gamma$, such that
    	\begin{equation}
        	\int_{S_\gamma} e^{\frac{r}{6}\int_{\T} \abs{u}^6\,\de x
            } \, \de\mu(u)<\infty.
    	\end{equation}
	\end{lemma}
	\begin{proof}
	    One can adapt the proof found in Section 4.1 of \cite{OST} or use the variational approach introduced in \cite{BG20} (see Proposition 3.1 in \cite{TolomeoWeber} for details).
        \end{proof}
    Next, we introduce a crucial stability lemma.
	\begin{lemma}[\cite{OST}, Lemma 6.3] \label{stability}
    	Given any $\varepsilon>0$ and $\delta^\ast>0$, there exists $\gamma=\gamma(\varepsilon,\delta^\ast)>0$ such that the following holds: suppose $u\in L^2(\T)$ satisfies $M(u)\le \norm{Q}_{L^2(\R)}^2$ and
    	\begin{equation}
    	   \big\lVert u-e^{i\theta}Q_{\delta,x_0}^\chi\big\rVert_{L^2(\T)}\ge\varepsilon
    	\end{equation}
    	for all $0<\delta<\delta^\ast$, $x_0\in\T$, and $\theta\in\R$. Then $u\in S_\gamma$.
	\end{lemma}
	\subsection{A change-of-variable formula}
    
	Let $U_{\varepsilon,\delta^\ast}=U_{\varepsilon,\delta^\ast}(\T)$ be an approximate $\varepsilon$-neighbourhood of $\mathcal{M}_{\delta^\ast}$ given by
	\begin{equation*}
    	U_{\varepsilon,\delta^\ast}=\big\{u\in L^2(\T):\big\lVert u-e^{i\theta}Q^\chi_{\delta,x_0}\big\rVert_{L^2(\T)}<\varepsilon \text{ for some } 0<\delta<\delta^\ast,x_0\in\T,\theta\in\R\big\}.
	\end{equation*}
    We say ``approximate'' because of the presence of the cutoff function $\chi$. \medskip
    
    In view of \Cref{S_gamma integrability} and \Cref{stability}, it suffices to prove that
	\begin{equation} \label{goalnbhd}
    	\int_{U_{\varepsilon,\delta^\ast}} 
        e^{
        \frac16 \norm{u}_{L^6(\T)}^6 + q \norm{u}_{L^6(\T)}^{6\alpha}
        }
        \mathbbm{1}_{ \{ M(u) \le \norm{Q}_{L^2(\R)}^2 \} } \, \de\mu(u) 
        < \infty
	\end{equation}
	for any small $\varepsilon,\delta^\ast>0$. To do this, we introduce a change-of-variable formula used by the authors of \cite{OST} for integrals over $U_{\varepsilon,\delta^\ast}$, which parametrises this neighbourhood of $\mathcal{M}_{\delta^\ast}$. Given $0 < \delta \ll 1$, $x_0 \in \T$ and $0 \leq \theta < 2\pi$, we define $V_{\delta,x_0,\theta} = V_{\delta,x_0,\theta}(\T)$, a subspace orthogonal to $\mathcal{M}_{\delta^\ast}$, by
    \begin{equation} \label{ogV}
       V_{\delta,x_0,\theta}(\T) \coloneqq
        \big\{
            u \in L^2(\T) :
            \innerprod{u}{(1- \partial_x^2) \varphi}_{L^2(\T)} =0
            \text{ for }
            \varphi \in 
                T \mathcal{M}_{\delta,x_0,\theta}(\T) 
        \big\},
    \end{equation}
    where $T \mathcal{M}_{\delta,x_0,\theta}(\T) \coloneqq \{ e^{i\theta} \partial_\delta Q_{\delta,x_0}^\chi, e^{i\theta} \partial_{x_0} Q_{\delta,x_0}^\chi, ie^{i\theta} Q_{\delta,x_0}^\chi\}$.
    Letting $\tau_{x_0}$ denote translation by $x_0 \in \T$, one can show that $V_{\delta,x_0,\theta}$ is a real vector space of codimension 3 in $L^2(\T)$, orthogonal (with the weight $(1- \partial_x^2)$) to the tangent vectors $e^{i\theta} \partial_\delta Q_{\delta,x_0}^\chi$, $e^{i\theta}\partial_{x_0} Q_{\delta,x_0}^\chi = e^{i\theta} \partial_{x_0} (\tau_{x_0}(\chi Q_\delta))$ and $\partial_\theta (e^{i \theta} Q_{\delta,x_0}^\chi) = i e^{i\theta}Q_{\delta,x_0}^\chi$ of the soliton manifold $\mathcal{M}_{\delta^\ast}$. \medskip

	The following proposition shows that a small neighbourhood of $\mathcal{M}_{\delta^\ast}$ can be endowed with an orthogonal coordinate system in terms of $e^{i\theta} \partial_\delta Q_{\delta,x_0}^\chi$, $e^{i\theta}\partial_{x_0} Q_{\delta,x_0}^\chi$, $i e^{i\theta}Q_{\delta,x_0}^\chi$ and $V_{\delta,x_0,\theta}$.
	Recalling that $\mu$ is supported on functions $u \in H^s(\T)$ for $s<\frac{1}{2}$ and that $\mu(H^\frac{1}{2}) = 0$, we redefine $U_{\varepsilon,\delta^*}$ to mean $U_{\varepsilon,\delta^*} \cap (H^{\frac{1}{2}})^c$. Next, given a small $\delta >0$ and $x_0 \in \T$ we let $V_{\delta,x_0,0} = V_{\delta,x_0,0} \cap H^1(\T)$. Let $\mu_{\delta,x_0}^\perp$ denote the Gaussian measure with $V_{\delta,x_0,0} \subset H^1(\T)$ as its Cameron-Martin space. Then we have the following change of variables for the $\mu$-integration over $U_{\varepsilon,\delta^*}$.
	\begin{lemma}[\cite{OST}, Lemma 6.10] \label{changeofvar}
        Fix $K>0$ and sufficiently small $\delta^\ast>0$.\footnote{The choice of $\delta^\ast$ comes from Proposition 6.4 in \cite{OST}, which introduces an orthogonal coordinate system for $U_{\varepsilon,\delta^\ast}$.} Let $F(u)\ge 0$ be a functional of $u$ on $\T$ which is continuous in some topology $H^s(\T)$, $s<\frac12$, such that $F\le C$ for some $C$ and $F(u)=0$ if $M(u)>K$. Then there is a locally finite measure $\sigma$ on $(0,\delta^{\ast})$ such that
    	\begin{align*} 
            	\int_{U_{\varepsilon,\delta^\ast}} F(u)\, \de\mu(u) \le  &\iiiint_{\widetilde{U}_{\varepsilon,\delta^\ast}} F(e^{i\theta}(Q^\chi_{\delta,x_0}+v)) \\
                &\phantom{\iint}\cdot e^{-\frac12\lVert Q^\chi_{\delta,x_0}\rVert_{H^1(\T)}^2-\innerprod{(1-\partial_x^2)Q^\chi_{\delta,x_0}}{v}_{L^2(\T)}
                }
                \, \de\mu^\perp_{\delta,x_0}(v)\, \de\sigma(\delta)\, \de x_0\, \de\theta,
        \end{align*}
    	where
    	\[
        	\widetilde{U}_{\varepsilon,\delta^\ast}=\{(\delta,x_0,\theta,v)\in(0,\delta^\ast)\times\T\times(\R/2\pi\Z)\times V_{\delta,x_0,0}:e^{i\theta}(Q^\chi_{\delta,x_0}+v)\in U_{\varepsilon,\delta^\ast}\}.\footnote{Here we are using an abuse of notation: $V_{\delta,x_0,0}$ is as in \eqref{ogV} with $\theta=0$, and contains the support of $\mu^\perp_{\delta,x_0}$.}
    	\]
    	Moreover, $\sigma$ is absolutely continuous with respect to the Lebesgue measure $\de\delta$ on $(0,\delta^\ast)$ with $|\frac{\de\sigma}{\de\delta}(\delta)|\lesssim\delta^{-20}$.
	\end{lemma}

    \subsection{Proof of (\ref{goalnbhd})} Next, we want to apply \Cref{changeofvar} to the integral in \eqref{goalnbhd}. The integrand 
	\[
        e^{ \frac16\norm{u}_{L^6(\T)}^6+q\norm{u}_{L^6(\T)}^{6\alpha} }
        \mathbbm{1}_{\{M(u)\le\norm{Q}_{L^2(\R)}^2\}}
	\]
	is neither bounded, nor continuous. However, the lack of continuity is due to the sharp cutoff $\mathbbm{1}_{\{M(u)\le\norm{Q}_{L^2(\R)}^2\}}$ which can be approximated by a smooth cutoff. We can then pass to the limit. We can also replace the integrand by a bounded one by taking 
	\[
        \mleft( \frac16\norm{u}_{L^6(\T)}^6+q\norm{u}_{L^6(\T)}^{6\alpha} \mright) \wedge M
	\]
	inside the exponential and letting $M \to \infty$, as long as we obtain bounds that are uniform in $M$. We will omit these approximation arguments. \medskip

	Applying \Cref{changeofvar} to the integral in \eqref{goalnbhd} we obtain
	\begin{align}
        \int_{U_{\varepsilon,\delta^\ast}} &e^{\frac16\norm{u}_{L^6(\T)}^6+q\norm{u}_{L^6(\T)}^{6\alpha}
        }
        \mathbbm{1}_{\{M(u)\le \norm{Q}_{L^2(\R)}^2\}}\, \de\mu(u) \notag
        \\
        &\leq    \iiiint_{\widetilde{U}_{\varepsilon,\delta^\ast}} e^{I(v)} 
        \mathbbm{1}_{\{M(Q^\chi_{\delta,x_0}+v) \le \norm{Q}_{L^2(\R)}^2\}}
        \, \de\mu^\perp_{\delta,x_0}(v)\, \de\sigma(\delta)\, \de x_0\, \de\theta \notag
        \\
        &=
        2\pi\iint_{\widetilde{U}_{\varepsilon,\delta^\ast}} e^{I(v)} 
        \mathbbm{1}_{\{M(Q^\chi_{\delta}+v)\le \norm{Q}_{L^2(\R)}^2\}}
        \, \de\mu^\perp_{\delta}(v)\, \de\sigma(\delta), \label{2pi}
	\end{align}
    where we used that the surface measure is translation invariant and where $\mu^\perp_{\delta} = \mu^\perp_{\delta,0}$ and
	\[
	    I(v) = \frac16\norm{Q_{\delta}^\chi + v}_{L^6(\T)}^6+q\norm{Q_{\delta}^\chi + v}_{L^6(\T)}^{6\alpha} - \frac{1}{2} \norm{Q_{\delta}^\chi}_{H^1(\T)}^2 - \innerprod{(1-\partial_x^2) Q_{\delta}^\chi}{v}.
	\]
	Observe that 
    $q \norm{Q_\delta^\chi+v}_{L^6(\T)}^{6\alpha} 
    \le 
    c_{q,\alpha} (\norm{Q_\delta^\chi}_{L^6(\T)}^{6\alpha} 
    + 
    \norm{v}_{L^6(\T)}^{6\alpha})$
    (with $c_{q,\alpha}\to 0$ as $q\to 0$) and
    $\norm{Q_{\delta}^\chi}_{L^6(\T)}^{6\alpha} 
    \le 
    \delta^{-2 \alpha} \norm{Q}_{L^6(\R)}^{6 \alpha}$.
    Using \eqref{2pi} with this observation, repeating (6.6.3) to (6.6.9) in \cite{OST}, and letting $B_{\varepsilon_1} \coloneqq \{ \norm{v}_{L^2(\T)} \leq \varepsilon_1 \}$, we arrive at
	\begin{align*}\label{formula(6.6.8)}
        &\int_{U_{\varepsilon,\delta^\ast}} e^{\frac16\norm{u}_{L^6(\T)}^6
        +
        q \norm{u}_{L^6(\T)}^{6\alpha}
        }
        \mathbbm{1}_{\{M(u)\le \norm{Q}_{L^2(\R)}^2\}}\, \de\mu(u)
        \\
        &\lesssim \int_0^{\delta^*} e^{\tilde{c}_{q,\alpha}\delta^{-2 \alpha}} \int_{B_{\varepsilon_1}} e^{c_{q,\alpha}\norm{v}_{L^6(\T)}^{6\alpha}} e^{J(v)} e^{C_\eta \norm{v}_{L^6(\T)}^6} \mathbbm{1}_{\{M(Q^\chi_{\delta}+v)\le \norm{Q}_{L^2(\R)}^2\}}
        \, \de\mu^\perp_{\delta}(v)\, \de\sigma(\delta)
        \\
        &\lesssim \int_0^{\delta^*} e^{\tilde{c}_{q,\alpha}\delta^{-2 \alpha}}\int_{B_{\varepsilon_1}}  e^{J(v)} e^{\widetilde{C}_\eta \norm{v}_{L^6(\T)}^6} \mathbbm{1}_{\{M(Q^\rho_{\delta}+v)\le \norm{Q}_{L^2(\R)}^2\}}
        \, \de\mu^\perp_{\delta}(v)\, \de\sigma(\delta),
	\end{align*}
	where $\tilde{c}_{q,\alpha} = c_{q,\alpha} \norm{Q}_{L^6(\T)}^{6\alpha}$, $\widetilde{C}_\eta = C_\eta + c_{q,\alpha}$, and $J(v)$ is defined by
	\[
        J(v) = (2\delta^{-2} -1) \innerprod{Q_\delta^\chi}{\Rea v}_{L^2(\T)} + (1 + \eta) \innerprod{( Q_\delta^\chi)^4}{2(\Rea v)^2 + \frac{1}{2} \abs{v}^2}_{L^2(\T)}.
	\]
	Here, $\varepsilon_1$ is a small number chosen in \Cref{integrabilityOrthogonal} below.\footnote{This $\varepsilon_1$ determines $\varepsilon$ and $\delta^\ast$ through Proposition 6.4 in \cite{OST}.} 
	\begin{lemma}[\cite{OST}, Lemma 6.12]\label{integrabilityOrthogonal}
        Given any $C'_\eta >0$, there exists a small $\varepsilon_1 >0$ such that
        \[
            \int_{\{ \norm{v}_{L^2(\T)} \leq \varepsilon_1 \}} e^{C'_\eta \int_{\T} \abs{v}^6 \, \de x} \, \de\mu^\perp_{\delta}(v) < \infty
        \]
        uniformly in $0 < \delta \ll 1$.
	\end{lemma}
	\noindent Applying Hölder's inequality and \Cref{integrabilityOrthogonal}, we obtain
	\begin{align*}\label{formulaAlmostThere}
        \int_{U_{\varepsilon,\delta^\ast}} &e^{\frac16\norm{u}_{L^6(\T)}^6+q\norm{u}_{L^6(\T)}^{6\alpha}}\mathbbm{1}_{\{M(u) \le \norm{Q}_{L^2(\R)}^2\}}\, \de\mu(u)
        \\
        &\lesssim \int_0^{\delta^*} e^{\tilde{c}_{q,\alpha}\delta^{-2 \alpha}} 
        \left(
        \int_{B_{\varepsilon_1}}  e^{(1+\eta)J(v)} \mathbbm{1}_{\{M(Q^\chi_{\delta}+v)\le \norm{Q}_{L^2(\R)}^2\}}
        \, \de\mu^\perp_{\delta}(v) 
        \right)^{\frac{1}{1+\eta}}\, \de\sigma(\delta).
	\end{align*}
	Next, using Lemma 6.13 and Proposition 6.14 from \cite{OST}, given any sufficiently small $\eta>0$ there exists a constant $c = c(\eta)$ such that 
	\[
        \int_{B_{\varepsilon_1}}  e^{(1+\eta)J(v)} \mathbbm{1}_{\{M(Q^\rho_{\delta}+v) \le \norm{Q}_{L^2(\R)}^2\}}
        \, \de\mu^\perp_{\delta}(v) \lesssim e^{-c \delta^{-1}}
	\]
    uniformly in $0 < \delta \ll 1$. Applying this bound, we obtain
	\begin{equation*}
        \int_{U_{\varepsilon,\delta^\ast}} e^{\frac16\norm{u}_{L^6(\T)}^6+q\norm{u}_{L^6(\T)}^{6\alpha}} \mathbbm{1}_{\{M(u) \le \norm{Q}_{L^2(\R)}^2\}}\, \de\mu(u)
        \lesssim
        \int_0^{\delta^*} e^{\tilde{c}_{q,\alpha}\delta^{-2 \alpha}} e^{-c\delta^{-1}}
        \, \de\sigma(\delta).	    
	\end{equation*}
	Since $|\frac{\de\sigma}{\de\delta}(\delta)|\lesssim\delta^{-20}$ by \Cref{changeofvar}, we find that the right hand side is finite if $\alpha < \frac{1}{2}$ or $\alpha = \frac{1}{2}$ and $q$ sufficiently small. This concludes the proof of \eqref{goalnbhd}, and hence of the integrability condition \eqref{goal1}.
    \section{The \textnormal{gKdV} case}\label{gKdVcase}
    In this section, we will sketch the proof of \Cref{a.s. bounds NLS}-(ii) which is very similar to the proof of \Cref{a.s. bounds NLS}-(i). The deterministic local theory for \eqref{gKdV} can be found in Theorem 1.1 of \cite{ChapoutoKishimoto} and a substitute for the tail estimate in the Fourier-Lebesgue space setting is given by Lemma 4.1 in \cite{ChapoutoKishimoto}.
    Then, repeating the proof of \Cref{Orlicz tail bound}, we obtain that for $q > 0 $, $\alpha \in (0,1)$, $s < 1 - \frac{1}{p}$ and $F$ as in \eqref{F defn}, there exist $c = c(\alpha, q, s, p)$ and $M_0 = M_0(\alpha,q,s,p)$ such that for $M\geq M_0$
    \begin{equation}\label{Orlicz tail bound gKdV}
        \big\lVert \mathbbm{1}_{\{\norm{u_0}_{\mathcal{F}L^{s,p}(\T)} \geq M\}} \big\rVert_{F^*(L)(\mu)} \lesssim_{\alpha,q,s,p} e^{-cM^{2 \alpha}}.
    \end{equation}
    Using \eqref{Orlicz tail bound gKdV}, the invariance of the Gibbs measure $\rho$ under the flow of \eqref{gKdV}, and \Cref{better integrability} as in \Cref{Our strategy}, we conclude that for $2 < p < \infty$, there exists a $\rho$-measurable set $\Sigma _p \subset \bigcap_{s < 1 - \frac{1}{p}} \mathcal{F}L^{s,p}(\T)$ of full $\rho$-measure such that  solutions $u$ to \eqref{gKdV} with initial data $u_0 \in \Sigma_p$ satisfy for any $s < 1 - \frac{1}{p}$
    \begin{equation}\label{bound FourierLebesgue}
        \sup_{t \in [-T,T]} \norm{u(t)}_{\mathcal{F}L^{s,p}(\T)} \lesssim_{s,p,u_0} \log(2+T).
    \end{equation}
    Defining $\Sigma = \bigcap_{n = 1}^\infty \Sigma_{2 + \frac{1}{n}}$, which is a $\rho$-measurable set of full measure, and using the embedding $\mathcal{F}L^{s + \frac{a}{2},p} (\T) \hookrightarrow H^s(\T)$ for all $2 < p < \infty$ and $a > \frac{p-2}{p}$, we infer that $\Sigma \subset \bigcap_{s < \frac{1}{2}} H^{s}(\T)$ and that solutions $u$ to \eqref{gKdV} with $u_0 \in \Sigma$ satisfy
	\[
		\sup_{t \in [-T,T]} \norm{u(t)}_{H^s(\T)} \lesssim_{s, u_0} \log(2+T)
	\]
	for all $s < \frac{1}{2}$ and $T>0$. This concludes the proof of \Cref{a.s. bounds NLS}-(ii).
	\section*{Acknowledgements}
	The authors would like to thank Leonardo Tolomeo for his suggestion of the problem and help throughout the project. The authors would also like to thank the anonymous referee for helpful comments and Ruoyuan Liu for useful discussions.
	%

\bibliographystyle{amsplain}

\end{document}